\documentclass{article}%
\usepackage{amsfonts}
\usepackage{amsmath}
\usepackage{amssymb}
\usepackage{graphicx}
\usepackage[square,numbers]{natbib}%
\usepackage[affil-it]{authblk}
\usepackage[english]{babel}
\providecommand{\U}[1]{\protect\rule{.1in}{.1in}}
\newtheorem{theorem}{Theorem}

\newtheorem{lemma}[theorem]{Lemma}

\DeclareMathOperator\Imaginary{Im}
\DeclareMathOperator\Real{Re}

\newenvironment{proof}[1][Proof]{\noindent\textbf{#1.} }{\ \rule{0.5em}{0.5em}}

\title{Bounds of sub-triangles}
\author[1]{Amalia Adlerteg\thanks{Corresponding author.}}
\author[1]{Linus Carlsson}
\affil[1]{%
M\"{a}lardalen University, Sweden}
\begin{document}

\maketitle

\begin{abstract}
We show upper and lower bounds for angles in iterations of trisections of
certain triangulations.

\end{abstract}

\section{Introduction}
Mesh refinement has been of interest for several decades, see
e.g.~\citet{bansch1991local}, \citet{perdomo2013proving}, \citet{korotov2021improved}. To estimate the numerical solutions
in, for example, the finite element method, there is a need to understand how
the angles in each mesh-element can vary.
Lower and upper bounds have been established for the longest edge bisection method, see e.g.~\citet{gutierrez2007complexity}.
We will study the behaviour of the angles when using the \emph{longest edge trisection}(LE3),
which is an iterative method explained in the next section. We assumethat a triangle has the angles $\alpha$, $\beta$, and $\gamma$,
from smallest to largest.
When performing the LE3, we create a family of new triangles which
is denoted $\Gamma$.

The results in this paper is an extract of some of the results in Amalia Adlerteg's master thesis, to be presented later this year. We present the following findings:
\begin{itemize}
\item In the general case, we numerically and partially prove that
$$\gamma^{\prime}\leq\gamma\,\frac{\arccos\left(\frac{-5}{2\sqrt{7}}\right)}{\pi/3},$$ where $\gamma^{\prime}$ is the largest angle in any
triangle $T\in\Gamma$, see Theorem~\ref{thm:Omega1} and Theorem~\ref{thm:Omega2}.
\item Starting with $\alpha=\beta=\gamma=\pi/3$ it follows that any angle
$v\in T\in\Gamma$ satisfies\footnote{The lower estimate is given in \citep{perdomo2013proving}.} 
$$\arctan{\left(\frac{\sqrt{3}}{11}\right)}\leq v\leq\arccos{\left(\frac{-5}{2\sqrt{7}}\right)},$$ see Theorem~\ref{thm:Omega1} and Theorem~\ref{thm:Omega2}.
\item Starting with $\alpha=\beta=\pi/4$ and $\gamma=\pi/2$ it follows that
any angle \\
$v\in T\in\Gamma$ satisfies\footnote{The lower estimate is given in \citep{perdomo2013proving}.} 
$$\frac{3}{4}\,\arctan{\left(\frac{\sqrt{3}}{11}\right)}\leq v\leq\frac{3}{2}\,\arccos{\left(\frac{-5}{2\sqrt{7}}\right)},$$ see Theorem~\ref{thm:Omega1} and Theorem~\ref{thm:Omega2}.
\end{itemize}

\section{Description and preliminary results of the LE3-method}

\subsection{Longest edge trisection}
On the longest edge of any triangle there exists two points that divide the edge into three equal parts.
If we draw two lines from these points to the opposing vertex we get the longest edge trisection of this triangle.
This can be repeated on our new triangles until we are satisfied with the mesh refinement.

\subsubsection{Mathematical understanding}
We can describe any triangle as a complex number, $z$,
on the positive real and imaginary axes by letting the longest edge lay from the origin to the point $1+0i$ and let the shortest edge be the vector given by $z$.
In this paper we will refer to this description of any triangle as the normalized form of it.
This gives us a closed space where these triangles are defined,
namely the space of triangles given by $\Sigma = \{ z \ | \ \Imaginary(z) > 0,\ \Real(z)\leq 1/2,\ |z-1|\leq 1\} $.
We define $W_\textrm{L}$ as the conformal mapping that transforms the left triangle obtained by LE3 into its normalized form.
The conformal mappings $W_\textrm{M}$ and $W_\textrm{R}$ are defined similarly for the middle and right triangle respectively.
From a starting triangle defined by complex $z$, the orbit of that starting point,
$\Gamma_z$, contains the normalized triangles obtained from LE3 of all subsequent triangles.
There are three points in $\Sigma$ whose orbits consists of finite number of distinct points.
They are $\omega_1=\frac{1}{3}+\frac{\sqrt{2}}{3}i$,
$\omega_2=\frac{1}{3}+\frac{\sqrt{2}}{6}i$ and 
$\omega_3=\frac{4}{9}+\frac{\sqrt{2}}{9}i$ which all have the orbit $\Gamma_\omega=\{\omega_1,\omega_2,\omega_3\}$. 

\subsubsection{Conformal mappings}
To normalize a triangle we can use Möbius functions since they preserve angles.
The Möbius transformation that normalizes a triangle, defined by complex $z$,
is dependent on where $z$ is located on $\Sigma$.
For example, if $|z-2/3|\leq 1/3$ then the longest edge of the right triangle is located differently than if $z$ was outside of that area.
Thus, $W_\textrm{R}$ is a piecewise function defined on two areas and, with similar reasoning, both $W_\textrm{L}$ and $W_\textrm{M}$ are also piecewise functions. 

If $f(z)=\frac{az+b}{cz+d}$, then $f$ can be derived from the composition of these simple transformations:
\begin{itemize}
    \item translation : $f(z)=z+c$, $c$ is a complex constant.
    \item rotation : $f(z)=\lambda z$, $\lambda$ is a complex constant such that $|\lambda|=1$.
    \item magnification : $f(z)=\rho z$, $\rho$ is a real strictly positive number.
    \item inversion : $f(z)=1/z$.
\end{itemize}
\begin{figure}
    \centering
    \includegraphics[width=0.5\linewidth]{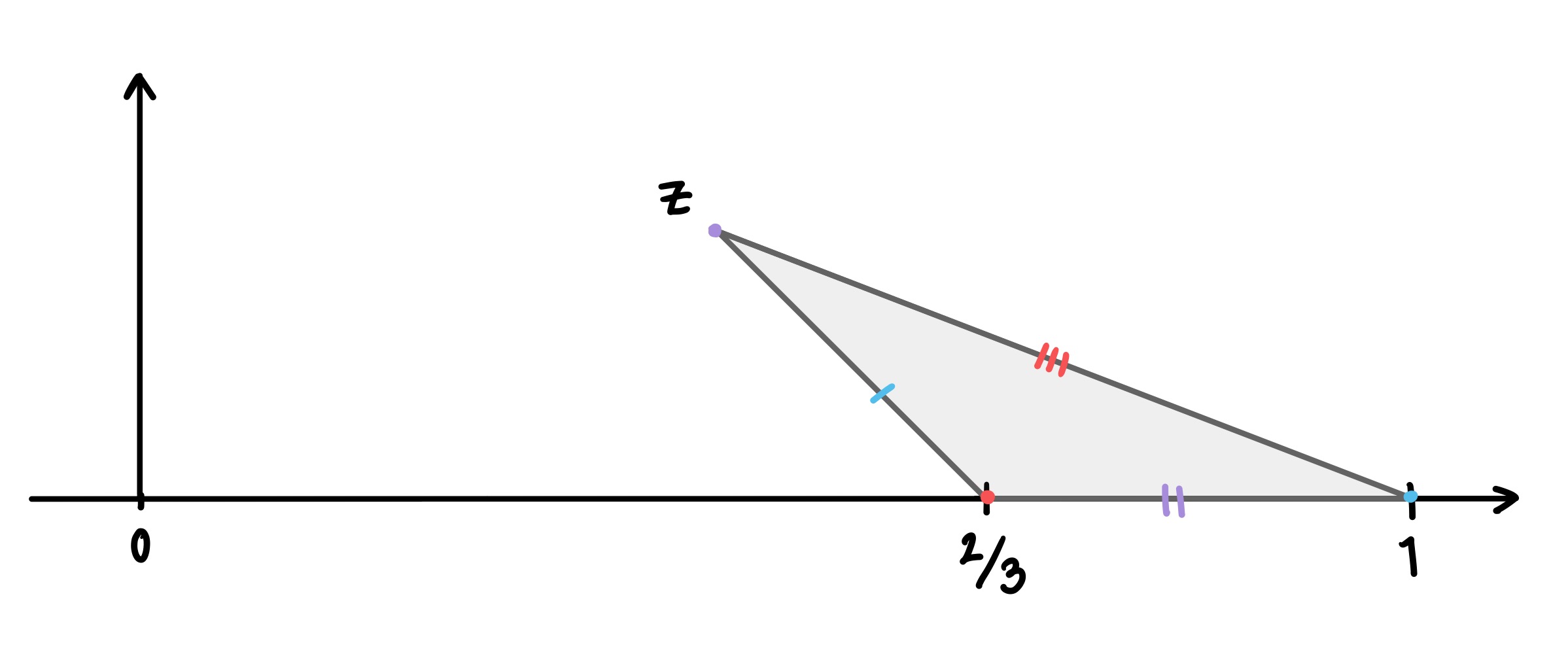} 
    \caption{A right triangle after trisection of the original triangle defined by $z$, such that $|z-2/3|\leq1/3$.}
    \label{fig:tr1}
\end{figure}
Say that $|z-2/3|\leq1/3$ and we want to normalize the right triangle illustrated in Figure~\ref{fig:tr1}. 
We can start with defining a translation that moves the point $z$ to the origin, that is 
\begin{equation}\label{eq:translation}
    f_1(\zeta) = \zeta - z. 
\end{equation} 
\begin{figure}[h]
    \centering
    \includegraphics[width=0.5\linewidth]{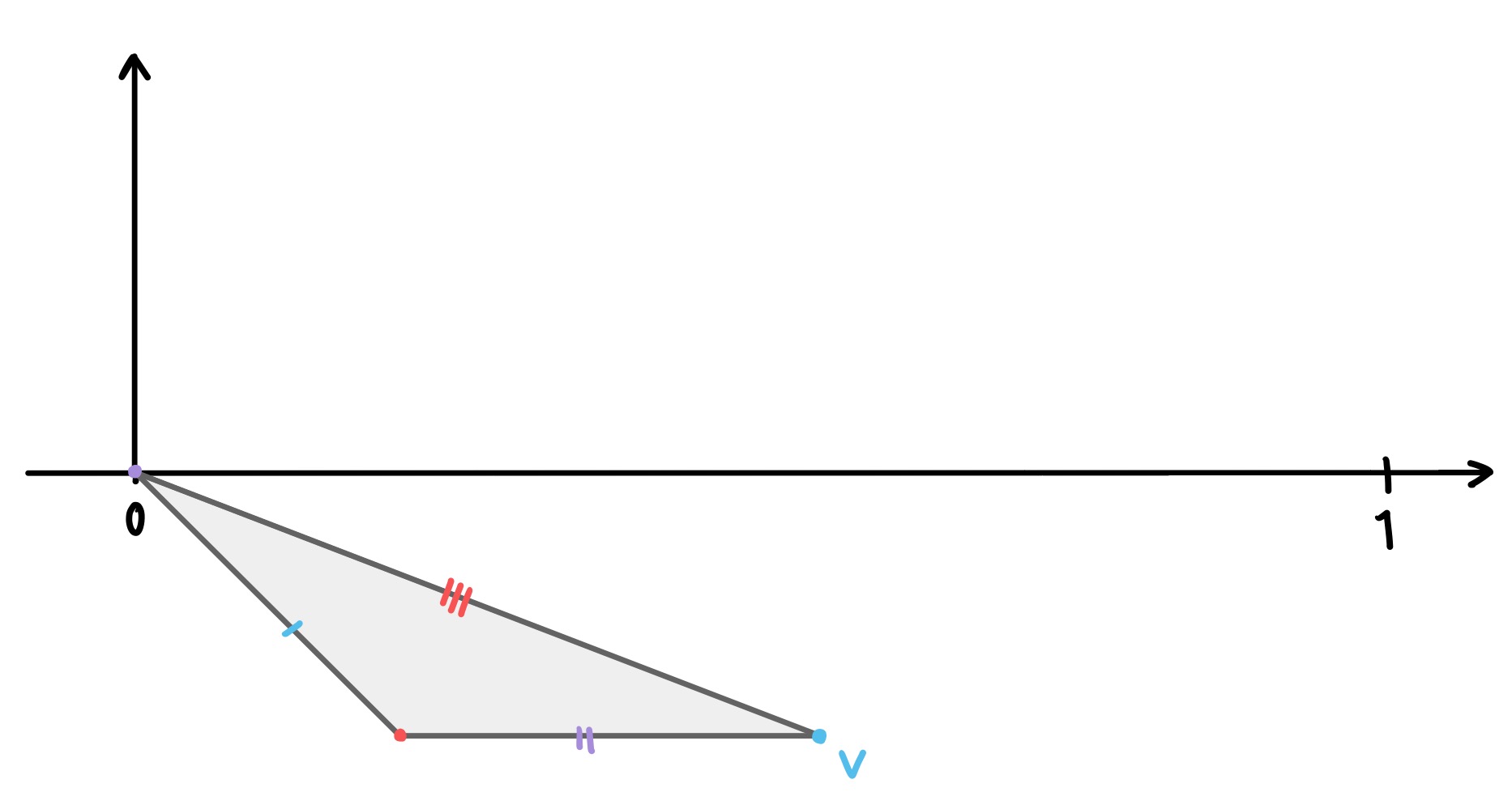} 
    \caption{The triangle after the translation given by Equation~\eqref{eq:translation}.}
    \label{fig:tr2}
\end{figure}
Then we define a rotation that moves the point $v$, illustrated in Figure~\ref{fig:tr2}, to the real axis.
The resulting triangle is displayed in Figure~\ref{fig:tr3}. 
The rotation is given by 
\begin{equation}\label{eq:rotation}
    f_2(\zeta)=\zeta \cdot\frac{\overline{v}}{|v|} = \zeta \cdot\frac{\overline{f_1(1+0i)}}{|f_1(1+0i)|}=\zeta\cdot\frac{1-\overline{z}}{|1-z|}.
\end{equation}
\begin{figure}[h]
    \centering
    \includegraphics[width=0.5\linewidth]{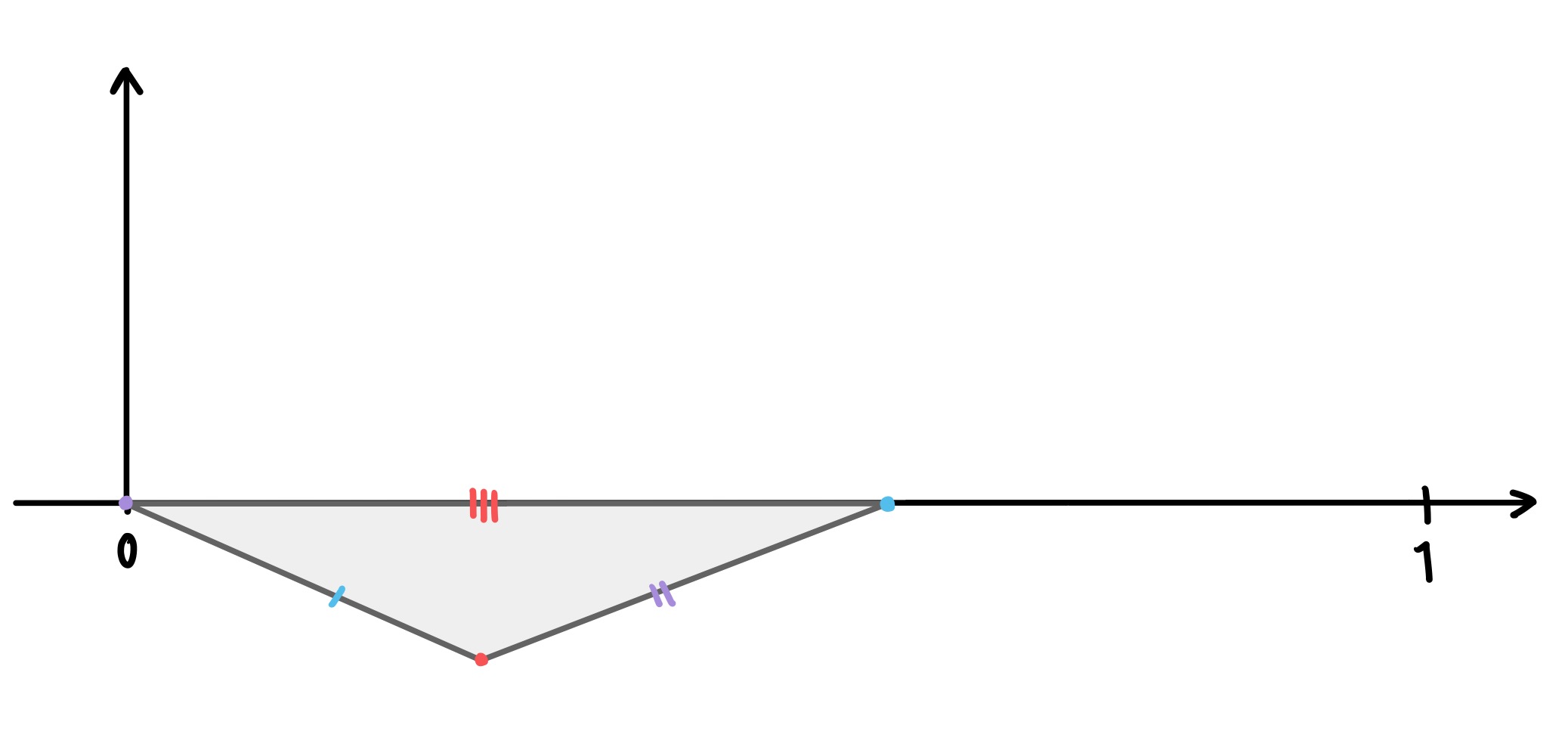} 
    \caption{The triangle after the rotation given by Equation~\eqref{eq:rotation}.}
    \label{fig:tr3}
\end{figure}
Now we define an inversion that mirrors all points on the real axis,
the result of which is illustrated in Figure~\ref{fig:tr4}. 
The inversion is given by taking the conjugate, that is
\begin{equation}\label{eq:inversion}
    f_3(\zeta) = \overline{\zeta}.
\end{equation}
\begin{figure}[h]
    \centering
    \includegraphics[width=0.5\linewidth]{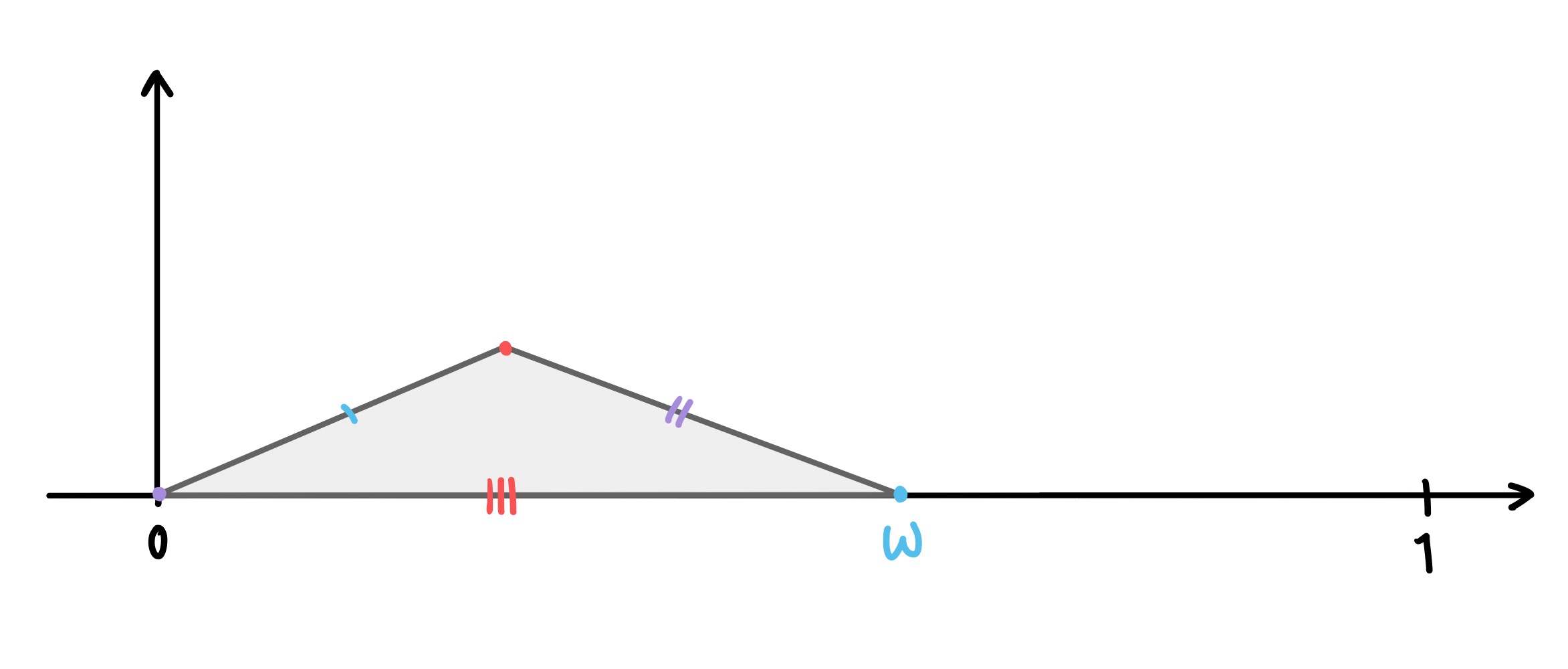} 
    \caption{The triangle after the inversion given by Equation~\eqref{eq:inversion}.}
    \label{fig:tr4}
\end{figure}
Finally we define a magnification that enlarges the triangle so that $w$, displayed in Figure~\ref{fig:tr4}, is at $1+0i$, that is
\begin{equation}\label{eq:magnification}
    f_4(\zeta)=\frac{\zeta}{|w|}= \frac{\zeta}{|f_3\circ f_2\circ f_1(1+0i)|}=\frac{\zeta}{|1-z|}.
\end{equation}
\begin{figure}[h]
    \centering
    \includegraphics[width=0.5\linewidth]{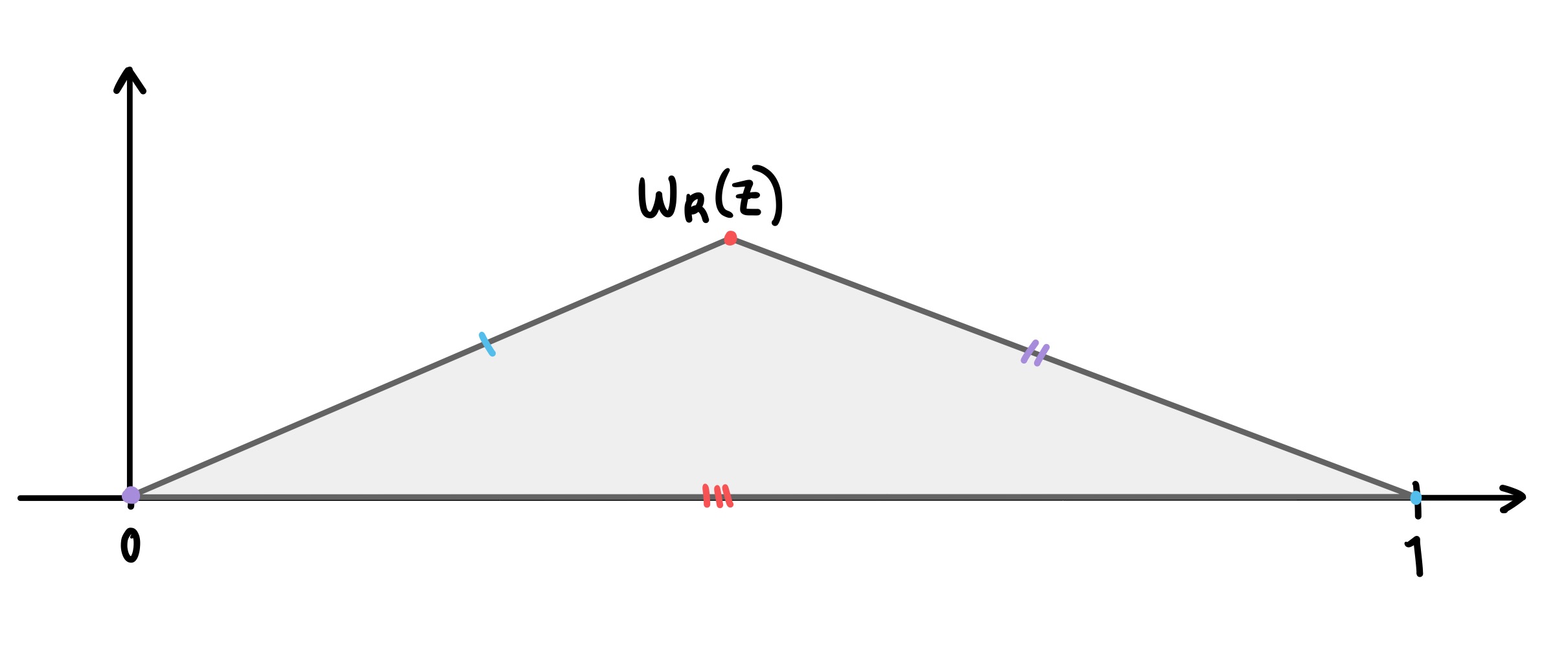} 
    \caption{The triangle after the magnification given by Equation~\eqref{eq:magnification}.}
    \label{fig:tr5}
\end{figure}
The final triangle after our simple transformations is illustrated in Figure~\ref{fig:tr5}. Now the complex number that describes the normalized triangle is originally at $2/3+0i$, so
$$W_\textrm{R}(z)=f_4\circ f_3\circ f_2\circ f_1\left(\frac{2}{3}+0i\right)=\frac{3\overline{z}-2}{3\overline{z}-3}.$$
For the other transformations see \citet{perdomo2013proving}.

\subsection{Poincar\'{e} half plane metric}
The metric of the Poincaré half-plane model on the half-plane $\{x+iy \ |\ y>0\}$ is given by
\begin{equation*}\label{Poincare metric}
    (ds)^2=\frac{(dx)^2+(dy)^2}{y^2},
\end{equation*}
where $s$ measures the local length along a line. The length of a curve is therefore given by
\begin{equation*}
    l(\gamma)=\int_\gamma ds.
\end{equation*}
With the parametrisation $\gamma:[0,1]\to \mathbb{C}$ such that $\gamma(t)=x(t)+iy(t)$, we get that
\begin{equation*}
    \frac{ds}{dt}= \frac{1}{dt}\,\sqrt{\frac{(dx)^2+(dy)^2}{y^2}}=\frac{1}{y}\,\sqrt{\frac{(dx)^2+(dy)^2}{(dt)^2}}=\frac{\sqrt{(\frac{dx}{dt})^2+(\frac{dy}{dt})^2}}{y}.
\end{equation*}
With this parametrisation the length of the curve is 
\begin{equation*}
    l(\gamma)=\int_\gamma ds = \int_0^1 \frac{ds}{dt}\,dt=\int_0^1 \frac{\sqrt{(\frac{dx}{dt})^2+(\frac{dy}{dt})^2}}{y}\,dt.
\end{equation*}
The hyperbolic distance, $d$, between $z_1$ and $z_2$ in the Poincaré half-plane is given by 
\begin{equation}\label{eq:hyperbolic distance}
    \cosh(d)= 1+\frac{|z_1-z_2|^2}{2\cdot\Imaginary(z_1)\cdot\Imaginary(z_2)}.
\end{equation}
In this space a geodesic is either a vertical line if the two points have the same real value, otherwise it is a circle arc that is centred on the real axis. 

Our derivations and simulations in this paper will rely heavily on the following Lemma, proved and presented as \emph{Lemma~3} in \citet{perdomo2013proving}. 
\begin{lemma}[Non-increasing property]\label{lm:non increasing}
    Let $d(\cdot,\cdot)$ denote the hyperbolic distance in the Poincaré half-plane. For every $z_1$ and $z_2$ in the space of triangles, $\Sigma$, we have
    $$d(W_j(z_1),W_j(z_2))\leq d(z_1,z_2), \hspace{2mm} j\in\{\textrm{L, M, R}\}.$$
\end{lemma}
We end this section by defining three hyperbolic circles, that is $C_1$, $C_2$ and $C_3$, with radius $\ln(\sqrt{2})$ and with hyperbolic centres $\omega_1$, $\omega_2$ and $\omega_3$, respectively.

\section{Simulations of orbits}
In this section we construct a fast method for orbit approximation, where one orbit is of special interest. 
\begin{lemma}\label{Lemma:orbit}
    The orbit of $z_\textrm{eq}=1/2+i\sqrt{3}/2$ is
    \begin{align*}
    \Gamma_{\textrm{eq}}=\mathcal{C}\bigcup \left\{ 
    \begin{array}{l}
        \frac{1}{2}+\frac{\sqrt{3}}{2}i,
        \frac{1}{14}+\frac{3\sqrt{3}}{14}i,
        \frac{1}{6}+\frac{\sqrt{3}}{6}i,
        \frac{5}{26}+\frac{3\sqrt{3}}{26}i,
        \frac{13}{42}+\frac{3\sqrt{3}}{14}i,\\
        \frac{5}{14}+\frac{3\sqrt{3}}{14}i,
        \frac{7}{18}+\frac{3\sqrt{3}}{18}i,
        \frac{29}{62}+\frac{3\sqrt{3}}{62}i,
        \frac{1}{2}+\frac{\sqrt{3}}{18}i,
        \frac{1}{2}+\frac{\sqrt{3}}{6}i
    \end{array} 
    \right\},
    \end{align*}
where $\mathcal{C}$ lies inside the union of circles $C_1,\,C_2$ and $C_3$.
\end{lemma}
The result in Lemma~\ref{Lemma:orbit} can be extracted from \citep{perdomo2013proving}.

\subsection{The method of random orbit simulation}
When calculating the orbit of a starting triangle, defined by complex $z_1$, the number of subsequent triangles grows rapidly.
This occurs since each triangle generates three new ones, 
thus already at the $10^{\textrm{th}}$ iteration there are too many points in the orbits to proceed further. 
To get a better understanding of longer orbits, we will use a random process iteration procedure.

We construct a singular orbit of the form $\Gamma_{z_1}^{1}=\{z_1,z_2,\hdots,z_N\}$ for ${z_{i+1}=W_j(z_i)}$, 
where $W_j$ is randomly chosen between $W_\textrm{L},W_\textrm{M}$ and $W_\textrm{R}$.
Since this orbit excludes many elements of the orbit for $z_1$, 
we approximate $\Gamma_{z_1}$ with $$\Gamma_{z_1}\approx \bigcup_{k=1}^{N}\Gamma_{z_1}^{k}.$$

\subsection{Examples\label{sec:examples}}
A simulation of the orbit of $z=\frac{2}{10}+\frac{1}{10}i$ is illustrated in Figure~\ref{fig:some_point}.
Here there are 1000 singular orbits and 1000 iterations.
\begin{figure}[h]
    \centering
    \includegraphics[width=0.6\linewidth]{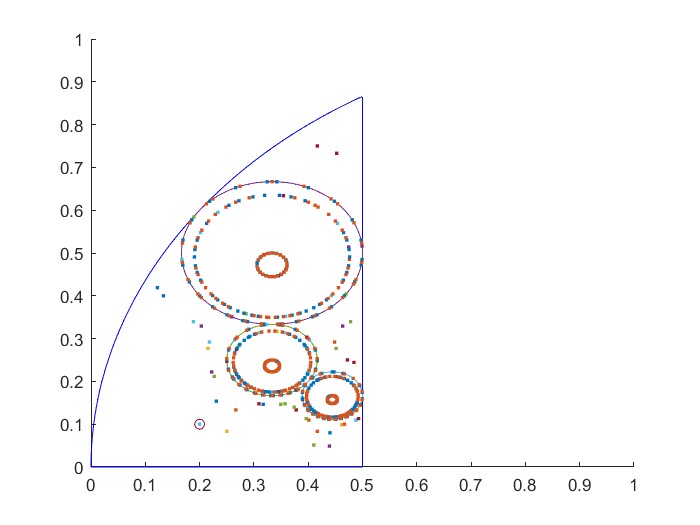}
    \caption{Simulation of the orbit starting at $z = \frac{2}{10}+\frac{1}{10}i$.}
    \label{fig:some_point}
\end{figure}
We can simulate the triangles given by complex numbers in $\Gamma_{\textrm{eq}}$ by starting at $z=\frac{1}{2}+\frac{\sqrt{3}}{2}i$, which is illustrated in Figure~\ref{fig:z_eq}.
\begin{figure}[h]
    \centering
    \includegraphics[width=0.6\linewidth]{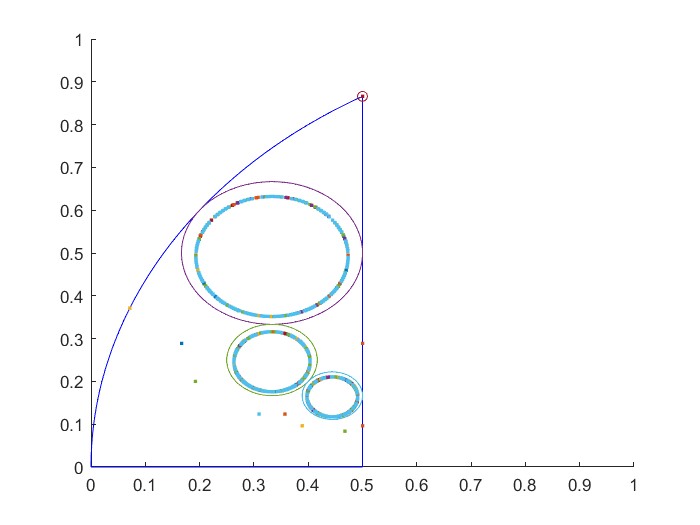}
    \caption{Simulation of the orbit starting at $z_{\textrm{eq}}= \frac{1}{2}+\frac{\sqrt{3}}{2}i$.}
    \label{fig:z_eq}
\end{figure}
In \citet{perdomo2023Similarity} they illustrate a figure, 
called \emph{Figure 10} in their work, 
of the orbit of $z=\frac{1}{10}+\frac{1}{10}i$ after ten iterations. 
With our method we can approximate these ten iterations of the orbit in 0.12 seconds, 
by simulating 10~000 singular orbits. 
The resulting orbit is displayed in Figure~\ref{fig:10_iterations} and seems to encapsulate the same pattern.
\begin{figure}[h]
    \centering
    \includegraphics[width=0.6\linewidth]{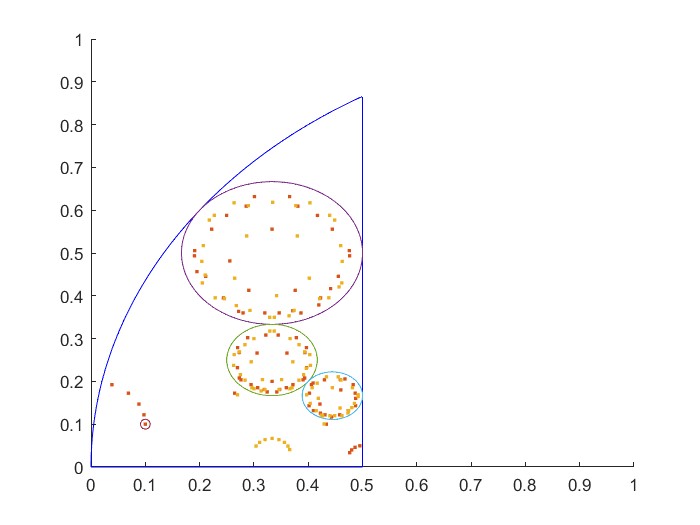}
    \caption{Simulation of the orbit starting at $z= \frac{1}{10}+\frac{1}{10}i$.}
    \label{fig:10_iterations}
\end{figure}

\section{Results}
We can divide the space of triangles, $\Sigma$, into two overlapping regions, $\Omega_1$ and $\Omega_2$. 
We trace the geodesics that tangent the circles $C_1$, $C_2$ and $C_3$ from the points in $\Gamma_{\textrm{eq}}$ that are outside of the circles.
The set $\Omega_1$ is then containing the circles $C_1$, $C_2$ and $C_3$ and the cones made up from the geodesics.
The region in $\Sigma$ under the circles $C_1,\, C_2$ and $C_3$ composes the set $\Omega_2$.
The following Lemma is proved and presented as \emph{Lemma 5} in \citet{perdomo2013proving}.
\begin{lemma}
    The set $\Omega_1$ is a closed region.
\end{lemma}
\begin{theorem}\label{thm:Omega1}
    Let $z$ be in $\Omega_1$ and let $z'$ be in the orbit of $z$,
    also let $\gamma$ and $\gamma'$ be the largest angles in the triangles defined by $z$ and $z'$ respectively.
    Then there exists a constant 
    $$c=\frac{\arccos\left(\frac{-5}{2\sqrt{7}}\right)}{\pi/3}\approx2.6816$$ such that $\gamma'\leq \gamma\cdot c$.
\end{theorem}
\begin{proof}
    For any $z\in\Omega_1$, the minimum largest angle, $\gamma$,
    is given at $z_\textrm{eq}$, meaning that $\gamma = \pi/3$.
    The maximum angle, $\gamma'$, is given at the triangle defined by $z_{\gamma'}=\frac{29}{62}+\frac{3\sqrt{3}}{62}i$
    and with use of the Pythagorean Theorem and the Cosine Theorem we get that $$\gamma'=\arccos\left(\frac{-5}{2\sqrt{7}}\right).$$
    Since $\Omega_1$ is a closed region,
    then from the largest quotient
    $c=\frac{\arccos\left(\frac{-5}{2\sqrt{7}}\right)}{\pi/3}$ in this set it follows that $\gamma'\leq \gamma\cdot c$.
\end{proof}
\begin{theorem}\label{thm:Omega2}
    Let $z$ be in $\Omega_2$ and let $z'$ be in the orbit of $z$, also let $\gamma$ and $\gamma'$ be the largest angles in the triangles defined by $z$ and $z'$ respectively.
    Then there exists a constant $c<2.1652$ such that $\gamma'\leq \gamma\cdot c$.
\end{theorem}
\begin{proof}
    For any $z$ in $\Omega_2$,
    we can find a radius $r>\ln(\sqrt{2})$ such that $z$ is on the line $l_r$ that is composed of the circle arcs in $\Omega_2$ with radius $r$ and centres in $\omega_1$,
    $\omega_2$ and $\omega_3$, see \citet{perdomo2013proving} for illustrations.
    For any $z'$ in the orbit of $z$, the distance from the points $\omega_1$,
    $\omega_2$ and $\omega_3$ to $z'$ is less than or equal to $r$, by Lemma~\ref{lm:non increasing}.
    Thus, for any $z$ the maximum quotient $\gamma'/\gamma$ would be obtained when $z_{\gamma'}$ and $z_\gamma$ lay on the curve $l_r$.
    Let $z_\gamma$ be the point on $l_r$ such that $|z_\gamma-1|=1$,
    which is clearly the point that acquires the smallest angle $\gamma$ on $l_r$.
    Also let $z_{\hat{\gamma}}=\frac{1}{2}+\frac{\sqrt{2}}{9}e^{-r}i$,
    that is $\Imaginary(z_{\hat{\gamma}}) = \min_{z\in l_r}(\Imaginary(z))$ and $\Real(z_{\hat{\gamma}})=1/2$.
    
    Since $\gamma'<\hat{\gamma}$, we can use this upper bound for $\gamma'$ instead.
    Now if we parametrize $z(t)$ along $|z-1|=1$ where $0<t\leq \arcsin(3/5)$,
    we can walk along all complex numbers $z\in\Omega_2$ that are on the arc $|z-1|=1$.
    Then we calculate $r$ using Equation~\eqref{eq:hyperbolic distance},
    where $d=d(z_\gamma,\omega_1)$.
    Finally, we can plot the quotient $\hat{\gamma}(t)/\gamma(t)$ for $0<t\leq \arcsin(3/5)$,
    which is displayed in Figure~\ref{fig:gamma kvot}.
    \begin{figure}[h]
        \centering
        \includegraphics[width=0.6\linewidth]{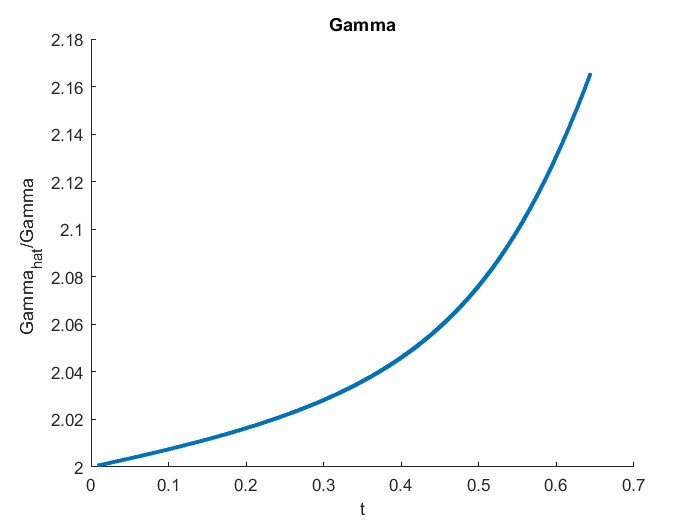}
        \caption{Plot of quotient $\hat{\gamma}(t)/\gamma(t)$ over $0<t\leq \arcsin(3/5)$.}
        \label{fig:gamma kvot}
    \end{figure}
    From this we numerically get that 
    $$\max_{0<t\leq\arcsin(3/5)}{\left(\frac{\hat{\gamma}(t)}{\gamma(t)}\right)}< 2.1652,$$ concluding our proof.
\end{proof}

\section{Conclusion and discussion}
The main results in this paper demonstrates that there exists an upper bound on the largest angle, $\gamma$,
that can occur in the orbit from LE3 of a given triangle.
For triangles defined by complex numbers in the region $\Omega_1$,
we proved analytically that any angle $\gamma'$ is bounded by the original largest angle $\gamma$ by
$$\gamma^{\prime}\leq\gamma\,\frac{\arccos\left(\frac{-5}{2\sqrt{7}}\right)}{\pi/3}.$$
Furthermore, in $\Omega_2$ we showed numerically that any angle $\gamma'$ is bounded by the original largest angle $\gamma$ by 
$$\gamma'<2.1652\gamma.$$
This ensures that an upper bound for the largest original angle in longest edge trisection exists.

\bibliography{myBib}

\end{document}